\documentclass[a4paper,12pt]{article}
\usepackage{amsmath}
\usepackage{amssymb}
\usepackage{setspace}
\usepackage{fullpage}
\usepackage{bbm}
\usepackage{amsthm}
\usepackage[latin1]{inputenc}
\usepackage{xcolor}
\usepackage[normalem]{ulem}
\usepackage{hyperref}
\usepackage{rotating}
\usepackage{graphicx}
\usepackage{enumerate}
\usepackage{dsfont}
\usepackage{pdfpages} 
\usepackage{mathrsfs}
\usepackage{cleveref}
\usepackage{comment}

\usepackage[normalem]{ulem} 

\usepackage[
singlelinecheck=false % <-- important
]{caption}

\newtheorem{theorem}{Theorem}[section]
\newtheorem{lemma}[theorem]{Lemma}

\newtheorem{corollary}[theorem]{Corollary}

\theoremstyle{definition}

\newtheorem{remark}[theorem]{Remark}
\newtheorem*{remark*}{Remark}

\newtheorem*{example*}{Example}

\usepackage{pdfsync}
\def\PG{\mathrm{PG}}

\title{The characterization of cones as pointsets with 3 intersection numbers}

\author{Dibyayoti Dhananjay Jena \thanks{Supported by the Marsden Fund Council administered by the Royal Society of New Zealand.}}
\date{}

\begin{document}
\maketitle

\begin{abstract}
	In (\cite{innamorati}), Innamorati and Zuanni have provided a combinatorial characterization of Baer and unital cones in $\PG(3,q)$. The current paper generalizes these results to arbitrary dimension. Furthermore, these results are extended to hyperoval and maximal arc cones.
\end{abstract}

\section{Introduction}
	The characterization of point sets in projective spaces by their intersection sizes with respect to subspaces has always been an interesting field of research in combinatorics. Some of the classical objects such as (Baer) subgeometries, Hermitian varieties, quadrics, conics, maximal arcs fall into the category of point sets with few intersection numbers. 

	Formally, a set of points $K$ in a projective space $\Sigma=\PG(n,q)$ is said to be of \textit{class} $[m_1,m_2,\dots,m_k]_d$ if every $d$-dimensional subspace of $\Sigma$ intersects $K$ in $m_i$ points for some $i\in\{1,2,\dots,k\}$. Moreover, $K$ is said to be a set of \textit{type} $(m_1,m_2,\dots,m_k)_d$ if it is of class $[m_1,m_2,\dots,m_k]_d$ and for every $i\in \{1,2,\dots,k\}$, there exists a $d$-dimensional subspace intersecting it in $m_i$ points. %Given a point set $K$ of type $(m_1,m_2,\dots,m_k)_d$, we abuse the notation to call a $d$-dimensional subspace `\textit{of type} $m_i$' if it intersects $K$ in $m_i$ points. 
	The numbers $m_i$ are simply called the \textit{intersection numbers} of $K$ with respect to the $d$-dimensional subspaces. See \cite{declerck} for more information about the characterization and construction of such sets.

	A set of points $K$ is called an $(n-d)$-\textit{blocking set} if every $d$-dimensional subspaces of $\Sigma$ has non-empty intersection with $K$. When $d=1$ and 2 the blocking sets are simply called \textit{line blocking sets} and \textit{plane blocking sets} respectively. A point $P$ in an $(n-d)$-blocking set $K$ is called \textit{essential} if $K\setminus\{P\}$ is no longer an $(n-d)$-blocking set. An $(n-d)$-blocking set is called \textit{minimal} if no proper subset of it is an $(n-d)$-blocking set itself or in other words if all its points are essential. Throughout the paper the number of points in an $n$-dimensional projective space $\PG(n,q)$ is denoted by $\theta_n=\frac{q^{n+1}-1}{q-1}$. The lower bound on the size of an $(n-d)$-blocking set is $\theta_{n-d}$ and is attained only when $K$ is an $(n-d)$-dimensional subspace of $\Sigma$ \cite{bose}. An $(n-d)$-blocking set is called \textit{non-trivial} if it does not contain any $(n-d)$-dimensional subspace. Non-trivial minimal blocking sets are of great interest to mathematicians and more about them can be found in \cite{blokhuis,blokhuis2,dipaola,szonyi}.

	Given a set of points in a projective space, it is easy to find its intersection numbers with different subspaces. The converse is however not always true. Identifying a geometric structure from its intersection numbers is a classical problem which is often difficult to solve. A lot of research has gone into characterizing sets with two intersection numbers (see \cite{batten,hamilton,napolitano,napolitano2,thas,ueberberg}). However less is known about sets with more than two intersection numbers (see \cite{batten2,coykendall,hirschfeld,segre,zannetti,zuanni}). An easy way of creating sets with three intersection numbers with respect to hyperplanes is by constructing \textit{cones} with base a set with two intersection numbers with respect to its hyperplanes and a subspace as the vertex. In this paper, we show that the converse is also true some special cases, i.e. in certain cases, a set with 3 intersection numbers must be a cone.

	A \textit{cone} $K$ in $\Sigma=\PG(n,q)$ with an $m$-dimensional vertex $V,V\subset K$, is a set of points such that for any pair of points $P\in K$ and $Q\in V$ all the points on the line $PQ$ lie in $K$. For any $(n-m-1)$-dimensional subspace $B$ of $\Sigma$ disjoint from $V$, the set $K\cap B$ is called a \textit{base of the cone}. Note that all the bases of a cone are isomorphic to each other. Therefore cones are usually named after their bases, e.g. a \textit{Baer cone} is a cone with a Baer subgeometry as the base, a \textit{unital cone} is a cone with an unital as the base. 

	For $q=p^h, t\mid h$, a subgeometry of order $p^t$ in $\Sigma=\PG(n,q)$ is a subset of points in $\Sigma$ such that there is a suitable frame with respect to which their coordinates belong to the subfield $\mathbb{F}_{p^t}$ of $\mathbb{F}_q$. When $h$ is even, a subgeometry of order $\sqrt{q}$ is called a \textit{Baer subgeometry}. For $n=1$ and $2$ it is called a \textit{Baer subline} and a \textit{Baer subplane} respectively. The following result was proved in \cite{innamorati}.
	\begin{theorem}\cite[Theorem 1.1.]{innamorati}\label{innzu1}
		In $\PG(3,q),$ with $q=p^h$ a square prime power, a 2-blocking set of type $(q+1,q+\sqrt{q}+1,\sqrt{q}^3+q+1)_2$ is a Baer cone.
	\end{theorem} 
	In section 2.1 this result will be generalized to a general cone with an even dimensional Baer subgeometry as the base and a subspace as the vertex.

	An \textit{embedded unital}  is a set of $\sqrt{q}^3+1$ points in a projective plane $\PG(2,q)$, $q=p^h$ a square prime power, such that every line intersects it in either $\sqrt{q}+1$ or $1$ point(s). In other words it is a set of $\sqrt{q}^3+1$ points of type $(1,\sqrt{q}+1)_1$ in $\PG(2,q)$. The following result was also proved in \cite{innamorati}.
	\begin{theorem}\cite[Theorem 1.2.]{innamorati}\label{innzu2}
		In $\PG(3,q),$ with $q=p^h$ a square prime power, a 2-blocking set of type $(q+1,\sqrt{q}^3+1,\sqrt{q}^3+q+1)_2$ is an unital cone.
	\end{theorem}  
	Section 2.2 generalizes this result to a general cone with a unital as base and a subspace as vertex.
	
	A \textit{maximal arc} of \textit{degree} $d$ in a projective plane $\Sigma=\PG(2,q)$ is a nonempty set of points $K$ such that every line of $\Sigma$ meets $K$ in either $0$ or $d$ points. It is not difficult to show that the size of a maximal arc is $qd+d-q$ and the degree $d$ of a maximal arc, distinct from the whole plane, always divides the order $q$ of the plane. In the special cases when $d=1$, a maximal arc is just a point, when $d=q$ it is the complement of a line and when $d=q+1$ it is the entire plane. A maximal arc is called \textit{trivial} if it is equivalent to one of the three aforementioned examples and \textit{non-trivial} otherwise. In \cite{ball} it was proven that non-trivial maximal arcs do not exist in Desarguesian projective planes of odd order, however when $q$ is even, maximal arcs of every admissible order are known to exist (see \cite{denniston}). 

	When $d=2$ a maximal arc is simply called a \textit{hyperoval}. A lot of research has gone into studying hyperovals and constructing them (see \cite{cherowitzo}, \cite{cherowitzo2}). In section 2.3 we explore cones with hyperovals as base and a subspace as vertex. Then the result has been generalized to maximal arc cones. However, unlike the previous cases, our characterization only holds for projective spaces with dimension at least five in case of maximal arc cones.
	
\section{Cones over some classical geometric objects}
	Throughout this section we assume that $K$ is a point set in $\PG(n,q)$ of type $(a,b,c)_{n-1}$. The number of points in $K$ is denoted by $k$. The hyperplanes intersecting $K$ in $i $ points are simply called $i$-\textit{hyperplanes} (or \textit{$i$-planes} if $n=3$) and the total number of $i$-hyperplanes in $\Sigma$ is denoted by $t_i$.

	With the notations above, standard double counting of the number of hyperplanes in $\Sigma$, the pairs $(P,H)$ with $P\in K$ and $H$, a hyperplane containing $P$, and the triples $(P_1,P_2,H)$ with $P_1,P_2\in K$ and $H$, a hyperplane containing $P_1$ and $P_2$, we get 
	\begin{equation}\label{1eq}
		t_a+t_b+t_c=\theta_n,
	\end{equation}
	\begin{equation}\label{2eq}
		at_a+bt_b+ct_c=k\theta_{n-1},
	\end{equation}
	\begin{equation}\label{3eq}
		a(a-1)t_a+b(b-1)t_b+c(c-1)t_c=k(k-1)\theta_{n-2}.
	\end{equation} 
	Solving these equations for $t_a,t_b$ and $t_c$, we get 
	\begin{equation}\label{ta}
		t_a=\dfrac{1}{(b-a)(c-a)}(k^2\theta_{n-2}-k(\theta_{n-2}+(b+c-1)\theta_{n-1})+bc\theta_n),
	\end{equation}
	\begin{equation}\label{tb}
		t_b=\dfrac{1}{(c-b)(a-b)}(k^2\theta_{n-2}-k(\theta_{n-2}+(a+c-1)\theta_{n-1})+ac\theta_n),
	\end{equation}
	\begin{equation}\label{tc}
		t_c=\dfrac{1}{(a-c)(b-c)}(k^2\theta_{n-2}-k(\theta_{n-2}+(a+b-1)\theta_{n-1})+ab\theta_n).
	\end{equation}
	Note that $t_a,t_b,t_c\ge1$ as $K$ is of type $(a,b,c)_{n-1}$. %In the following subsections we use this idea to show that some point sets with three different intersection numbers are cones.
 	
 	The following lemma provides an essential divisibility condition on $k$ using the intersection numbers $a,b$ and $c$.
 	\begin{lemma}\label{equivl}
 		If $K$ is a point set in $\Sigma=\PG(n,q)$ of type $(a,b,c)_{n-1}$ with $a\equiv b\equiv c\equiv \theta_{n-1}\equiv \alpha$(mod $\beta$) where $\alpha$ and $\beta$ are coprime, then $k\equiv\theta_n$(mod $\beta$).
 	\end{lemma}
 	
 	\begin{proof}
 		As $a\equiv b\equiv c\equiv \theta_{n-1}\equiv \alpha$(mod $\beta$), from \Cref{2eq} we have 
 		\begin{align*}
 			&k\theta_{n-1}\equiv at_a+bt_b+ct_c(\text{mod } \beta)\\
 			\Rightarrow&k\alpha\equiv (t_a+t_b+t_c)\alpha(\text{mod } \beta).
 		\end{align*}
 		Since $\alpha$ and $\beta$ are coprime, $\alpha$ is invertible in $\mathbb{Z}/\beta\mathbb{Z}$. So using \Cref{1eq} it follows that
 		\begin{align*}	
 			&k\equiv \theta_n(\text{mod } \beta).
 		\end{align*}
 	\end{proof}
 
\subsection{Baer cones}
	Let $\Sigma=\PG(n,q)$, where $q=p^h$ is a square prime power. For $-1\le r\le n,-1\le s\le n,r+s<n,$ an $(r,s)$-\textit{Baer cone} in $\Sigma$ is the set of points of a cone with an $r$-dimensional subspace as vertex and an $s$-dimensional Baer subgeometry as the base. The number of points in an $(r,s)$-Baer cone, denoted by $C_{r,s}$, is $$C_{r,s}=\dfrac{\sqrt{q}^{s+1}-1}{\sqrt{q}-1}q^{r+1}+\dfrac{q^{r+1}-1}{q-1}.$$
	It is easy to see that for $r\ge0,s\ge 1, r+s=n-1$, an $(r,s)$-Baer cone is a set of type $(a,b,c)_{n-1}$ with $a=C_{r,s-2},b=C_{r-1,s},c=C_{r,s-1}$ and blocks all lines of $\Sigma$. In this section we prove the converse.

	\begin{theorem}\label{baer2}
 		In $\Sigma=\PG(n,q)$, with $n\ge4, q\ge16$ and $q=p^h$ a square prime power, for $2\le2t\le n$ an $(n-t)$-blocking set $K$ of type $(a,b,c)_{n-1}$, where $a=C_{n-2t-1,2t-2},b=C_{n-2t-2,2t}$ and $c=C_{n-2t-1,2t-1}$, is an $(n-2t-1,2t)$-Baer cone.
 		
 		Moreover the statement is also true for $q\ge4$ if $t=1$.
	\end{theorem}

	The proof of the theorem has been split into the following five steps.

	\textbf{Step 1.} If $\pi$ is an $a$-hyperplane then $K\cap \pi$ contains an $(s,2(n-t-s-2))$-Baer cone for some $s$ with $n-2t-1\le s\le n-t-2$. 

	\begin{proof}\renewcommand{\qedsymbol}{}
		As $K$ is an $(n-t)$-blocking set, the restriction of $K$ to any subspace of $\Sigma$ also blocks the $t$-spaces contained in it. Hence for a hyperplane $\pi$, $K\cap\pi$ is an $(n-t-1)$-blocking set.
	
		Let $\pi$ be an $a$-hyperplane and $K'$ be a minimal $(n-t-1)$-blocking set in it. As $|K'|\le|K\cap\pi|= a=\theta_{n-t-1}+\theta_{t-2} \sqrt{q}q^{n-2t}\le\theta_{n-t-1} + \theta_{n-t-2} \sqrt{q}$, by \cite[Theorem 1.6]{boklerl} $K'$ contains an $(s,2(n-t-s-2))$-Baer cone for some $s$ with max$\{-1,n-2t-2\}\le s\le n-t-2$. However $K'$ cannot contain an $(n-2t-2,2t)$-Baer cone as otherwise $a= |K\cap\pi|\ge|K'|\ge C_{n-2t-2,2t}> a$, a contradiction. As $n\ge2t,$ max$\{-1,n-2t-1\}=n-2t-1$. Hence $\pi$ contains an $(s,2(n-t-s-2))$-Baer cone for some $s$ with $n-2t-1\le s\le n-t-2$.
		
		If $t=1$, since $|K\cap\pi|=a=\theta_{n-2}$, by \cite{bose} $K\cap \pi$ is an $(n-2)$-dimensional subspace of $\Sigma$. 
	\end{proof}

	\textbf{Step 2.} $K$ has at most $a+\sqrt{q}^{2n-4t+3}\theta_{t-1}$ points.

	\begin{proof}\renewcommand{\qedsymbol}{}
		Let $\pi$ be an $a$-hyperplane in $\Sigma$. Then by Step 1, $\pi$ contains an $(s,2(n-t-s-2))$-Baer cone for some $s$ with $n-2t-1\le s\le n-t-2$. Considering dimensions of the base and the vertex of the cone we see that in all the possible cases the cone is contained in an $(n-2)$-dimensional subspace. Considering the size of the cone we get that there exists an $(n-2)$-dimensional subspace $\omega$ of $\pi$ containing at least $\theta_{n-t-1}$ points of $K$. As any hyperplane through $\omega$ other than $\pi$ has at most $c$ points in it, we get 
		\begin{equation*}
			k\le a+q(c-\theta_{n-t-1})=a+\sqrt{q}^{2n-4t+3}\theta_{t-1}.
		\end{equation*} 
	\end{proof}

	\textbf{Step 3.} $K$ has 1(mod $q$) points.

	\begin{proof}\renewcommand{\qedsymbol}{}
		As $a\equiv b\equiv c\equiv \theta_{n-1}\equiv\theta_n\equiv1(\text{mod } q)$ from \Cref{equivl}, we have $k\equiv1(\text{mod }q)$.
	\end{proof}

	\textbf{Step 4.} $k<C_{n-2t-1,2t}+q=\theta_{n-t}+\sqrt{q}^{2n-4t+1}\theta_{t-1}+q.$ 

	\begin{proof}\renewcommand{\qedsymbol}{}
		Suppose to the contrary $k> C_{n-2t-1,2t}$. From \Cref{ta}, $t_a$ is a quadratic expression in $k$ with a positive leading coefficient as $a<b<c$. Therefore if $t_a$ is negative for two distinct values $k_0,k_1$ (with $k_0<k_1$) of $k$, then it is negative in the interval $[k_0,k_1]$.
	
		Evaluating $t_a$ for the given values of $a,b,c$ and $k=C_{n-2t-1,2t}+q$, we get 
		\begin{equation*}
			t_a=\theta_{t-1}(-q^{t+3/2}+q^{t+1}-q^{1/2}+2q^{-n+2t+2}-q^{-n+t+3/2})+\theta_{n-3}(q^{-2n+3t+3})
		\end{equation*}
		which is negative for all $n\ge4, 2\le2t\le n $ and $q\ge16$. For $t=1$, the expression above is negative for all $n\ge4, q\ge4$. 
			
		Evaluating $t_a$ for $k=a+\sqrt{q}^{2n-4t+3}\theta_{t-1}$, we get 
		\begin{align*}
			t_a=\dfrac{1}{(q-1)^3}&[-q^{n+t+5/2}+q^{n+t+2}+2q^{n+t+3/2}-q^{n+t+1}-3q^{n+t+1/2}+2q^{n+t-1/2}+q^{n+t-1}\\
			&+q^{n-t+2}+q^{n+5/2}-q^{n+2}-q^{n+1/2}-q^n+q^{2t+7/2}-q^{2t+5/2}-q^{2t+2}+2q^{2t+1}\\
			&+q^{2t+1/2}-q^{2t}-q^{2t-1/2}-q^{t+4}-q^{t+7/2}+q^{t+3}-q^{t+2}+q^{t+3/2}-q^{-t+4}\\
			&-q^{-t+3}+q^{-t+2}+2q^4-q^2+2q-1]
		\end{align*}
		which is negative for all $q\ge16.$ For $t=1$, the expression above is negative for all $n\ge4,q\ge4$.
	
		Therefore $t_a$ is negative for $k$ in $[C_{n-2t-1,2t}+q,a+\sqrt{q}^{2n-4t+3}\theta_{t-1}]$ for all $n\ge4, 2\le2t\le n $ and $q\ge16$ and for $t=1$ it is negative for all $n\ge4,q\ge4$. As from Step 2 $k\le a+\sqrt{q}^{2n-4t+3}\theta_{t-1}$, we have $k< C_{n-2t-1,2t}+q$. 
	\end{proof}

	\textbf{Step 5.} $K$ is an $(n-2t-1,2t)$-Baer cone.
	
	\begin{proof}
		By Steps 3 and 4 we see that $K$ has at most $C_{n-2t-1,2t}$ points. Let $K'$ be a minimal $(n-t)$-blocking set of $\Sigma$ contained inside $K$. Then $K'$ has at most $C_{n-2t-1,2t}\le \theta_{n-t}+\theta_{n-t-1}\sqrt{q}$ points. Thus by \cite[Theorem 1.6]{bokler} (and \cite[Theorem 1.5]{boklerl} for $t=1$ and $q\ge4$) $K'$ must contain an $(s,2(n-t-s-1))$-cone for some $s$ with max$\{-1,n-2t-1\}\le s\le n-t-1.$ Note that if $s\ge n-2t$, then there exists a hyperplane $H$ of $\Sigma$ containing the aforementioned cone. In this case $|K\cap H|\ge|K'\cap H|>q^{n-t}>c$, a contradiction to the assumption that every hyperplane meets $K$ in either $a,b$ or $c$ points. Hence the only possible cone contained in $K'$ is an $(n-2t-1,2t)$-cone. Comparing the sizes of $K$ and the cone $K'$, we see that $K$ is an $(n-2t-1,2t)$-cone. 
	\end{proof}

		Combining Theorems \ref{innzu1} and \ref{baer2}, we obtain the following corollary.
	\begin{corollary}\label{baergen}
		In $\Sigma=\PG(n,q)$, with $n\ge3$ and $q=p^{h}$ a square prime power, an $(n-1)$-blocking set $K$ of type $(a,b,c)_{n-1}$, where $a=C_{n-3,0},b=C_{n-4,2}$ and $c=C_{n-3,1}$, is an $(n-3,2)$-Baer cone.
	\end{corollary}

\subsection{Unital cones}
	We will now show that \Cref{innzu2} can be extended to unital cones with vertex of arbitrary dimension.
	\begin{theorem}\label{unital}
		 In $\Sigma=\PG(n,q)$, $n\ge4$ and $q$ a square prime power, an $(n-1)$-blocking set $K$ of type $(a,b,c)_{n-1}$, where $a=\theta_{n-2}, b=\theta_{n-3}+\sqrt{q}^{2n-3}$ and $c=\theta_{n-2}+\sqrt{q}^{2n-3}$, is a unital cone i.e. a cone with $(n-3)$-dimensional subspace as vertex and a unital as base.
	\end{theorem}

	The proof of the theorem has been split into the following ten steps.

	\textbf{Step 1.} If $\pi$ is an $a$-hyperplane then $K\cap\pi$ is an $(n-2)$-dimensional subspace of $\Sigma$. 

	\begin{proof}\renewcommand{\qedsymbol}{}
		As $K$ is a line blocking set, so is its restriction to any subspace of $\Sigma$. Hence, for the hyperplane $\pi$, $K\cap\pi$ is an $(n-2)$-blocking set. Since $|K\cap\pi|=a=\theta_{n-2}$, by \cite{bose}, the assertion follows.
	\end{proof}

	\textbf{Step 2.} $k\le\theta_{n-2}+\sqrt{q}^{{2n-1}}$.

	\begin{proof}\renewcommand{\qedsymbol}{}
		Let $\pi$ be an $a$-hyperplane and $\alpha$ be the $(n-2)$-space $K\cap\pi$. Each of the $q$ hyperplanes containing $\alpha$, other than $\pi$, have at most $c-a=\sqrt{q}^{2n-3}$ points outside $\alpha$. Hence $$k\le a+q\sqrt{q}^{2n-3}=\theta_{n-2}+\sqrt{q}^{2n-1}.$$  
	\end{proof}

	\textbf{Step 3.} $k\equiv\theta_{n-3}$(mod $q^{n-2}$).

	\begin{proof}\renewcommand{\qedsymbol}{}
		Note that $a\equiv b\equiv c\equiv \theta_{n-1}\equiv \theta_n\equiv\theta_{n-3}$(mod $q^{n-2}$). Recall that $q$ is a prime power. Therefore $\theta_{n-3}$ and $q^{n-2}$ are coprime and by \Cref{equivl}, we have $k\equiv\theta_{n-3}$(mod $q^{n-2}$).
	\end{proof}

%	\textbf{Step 4.} For any $0\le m\le n-2$, there exists an $m$-dimensional subspace of $\Sigma$ intersecting $K$ in at most $\theta_{m-1}+\sqrt{q}^{2m-1}$ points.  

%	\begin{proof}
%		We prove this by induction on $m$. 
	
%		Note that the base step of the induction follows trivially as there always exists a point of $\Sigma$ not in $K$.
	
%		Assume the statement to be true for any $m=m'$ and false for $m=m'+1$. Then all the $(m'+1)$-dimensional subspaces of $\Sigma$ intersect $K$ in at least $\theta_{m'}+\sqrt{q}^{2m'+1}$ points. Let $\beta$ be an $m'$-dimensional subspace intersecting $K$ in at most $\theta_{m'-1}+\sqrt{q}^{2m'-1}$. Counting the points of $K$ through the $(m'+1)$-spaces containing $\beta$ we get that 
%		\begin{align*}
%			k\ge &(\theta_{m'}+\sqrt{q}^{2m'+1}-\theta_{m'-1}-\sqrt{q}^{2m'-1})\theta_{n-m'-1}+\theta_{m'-1}-\sqrt{q}^{2m'-1}\\
%			>&\theta_{n-2}+\sqrt{q}^{2n-1},
%		\end{align*}   
%		a contradiction. Hence the statement holds true by induction.
%	\end{proof}

	\textbf{Step 4.} $k\ge \theta_{n-1}.$
	
	\begin{proof}\renewcommand{\qedsymbol}{}
		Let $\pi$ be an $a$-hyperplane and $\alpha$ be the $(n-2)$-space $K\cap\pi$. Consider an $(n-2)$-space $\gamma$ contained in $\pi$ other than $\alpha$. Then $K\cap\gamma$ contains $\theta_{n-3}$ points. Every hyperplane through $\gamma$ has at least $q^{n-2}$ points outside $\gamma$. Therefore $k\ge (q+1)q^{n-2}+\theta_{n-3}=\theta_{n-1}.$
	\end{proof}

	\textbf{Step 5.} $k=\theta_{n-2}+\sqrt{q}^{2n-1},t_a=\sqrt{q}^3+1,t_b=\theta_n-\theta_2$ and $t_c=q^2-\sqrt{q}^3+q$.
	
	\begin{proof}\renewcommand{\qedsymbol}{}
		From \Cref{tc}, we see that $t_c$ is a quadratic in $k$ with a positive leading coefficient as $a<b<c$. Therefore if $t_c$ is negative for two distinct values $k_0,k_1$ (with $k_0<k_1$) of $k$, then it is negative for any value in the interval $[k_0,k_1]$.
		
		Evaluating $t_c$ for the given values of $a,b,c$ and $k=\theta_{n-1}+q^{n-2},$ we get 
		\begin{equation*}
			t_c=\dfrac{1}{\sqrt{q}(q-1)}(-q^{n+1/2}+q^n+q^{n-1}+q^3-q^{5/2}+q^2+q^{3/2}-2q+q^{1/2}-2)
		\end{equation*}
		which is negative for all $n\ge4$ and $q\ge4$.
		
		Evaluating $t_c$ for the given values of $a,b,c$ and $k=\theta_{n-3}+\sqrt{q}^{2n-1}$, we get
		\begin{equation*}
			t_c=\dfrac{1}{(q-1)}(-q^{n}+q^{n-1/2}+q^{n-3/2}+q^3-q^{5/2}+q-q^{1/2}-1)
		\end{equation*}
		which is negative for all $n\ge 4$ and $q\ge 4$.
		
		Therefore $t_c$ is negative for $k$ in $[\theta_{n-1}+q^{n-2},\theta_{n-3}+\sqrt{q}^{2n-1}]$. This implies that either $k<\theta_{n-1}+q^{n-2}$ or $k>\theta_{n-3}+\sqrt{q}^{2n-1}$. By Steps 3 and 4 the former implies that $k=\theta_{n-1}$. But in this case, evaluating $t_b$ for the given values of $a,b$ and $c$, we see that $$t_b=-\sqrt{q}^3,$$ a contradiction. In the later case, using Step 2 and 3 we get $k=\theta_{n-2}+\sqrt{q}^{2n-1}$. The values of $t_a,t_b$ and $t_c$ now follow from Equations (\ref{ta}), (\ref{tb}) and (\ref{tc}) respectively.
 	\end{proof}

	\textbf{Step 6.} Let $\pi$ be an $a$-hyperplane and $\alpha$ be the $(n-2)$-space $K\cap\pi$. Let $u_i(\alpha)$ be the number of $i$-hyperplanes through $\alpha$. Then we have $u_a(\alpha)=1,u_b(\alpha)=0$ and $u_c(\alpha)=q.$

	\begin{proof}\renewcommand{\qedsymbol}{}
		Counting the hyperplanes through $\alpha$ we have
		\begin{equation}\label{adv1}
			u_a(\alpha)+u_b(\alpha)+u_c(\alpha)=q+1.
		\end{equation}
		And counting the points of $K$ contained in these hyperplanes we have
		\begin{equation}\label{adv2}
			\theta_{n-2}+(\sqrt{q}^{2n-3}-q^{n-2})u_b(\alpha)+\sqrt{q}^{2n-3}u_c(\alpha)=k.
		\end{equation}
		Solving Equations (\ref{adv1}) and (\ref{adv2}) for $k=\theta_{n-2}+\sqrt{q}^{2n-1}$, we get 
		\begin{equation}\label{advsol}
			\sqrt{q}u_a(\alpha)+u_b(\alpha)=\sqrt{q}.
		\end{equation}
		As $u_a(\alpha)\ge1$ and $u_b\ge0$, we have $u_a(\alpha)=1,u_b(\alpha)=0$ and $u_c(\alpha)=q.$
	\end{proof}

	\textbf{Step 7.} Let $S$ be the set of $(n-2)$-spaces defined by $$S:=\{K\cap\pi\mid \text{$\pi$ is an $a$-hyperplane}\}.$$ Then $|S|=\sqrt{q}^3+1$. Moreover for two distinct elements $\alpha_1,\alpha_2\in S$, dim$(\alpha_1\cap\alpha_2)=n-3$.

	\begin{proof}\renewcommand{\qedsymbol}{}
		By Step 6, we have that the $(n-2)$-dimensional spaces corresponding to different $a$-hyperplanes are distinct.  %If $\alpha_1=\alpha_2$, then $u_a(\alpha_1)=u_a(\alpha_2)\ge 2$ contradicting the fact that $u_a(\alpha_1)=u_a(\alpha_2)=1$. Hence $\alpha_1\ne\alpha_2$. 
		Therefore $|S|=t_a=\sqrt{q}^3+1$.
		
		Let $\pi_1$ and $\pi_2$ be two different $a$-hyperplanes with $\alpha_1=K\cap\pi_1$ and $\alpha_2=K\cap\pi_2$.
		As $\alpha_1$ is not contained in the hyperplane $\pi_2$, we have dim$(\alpha_1\cap\pi_2)=n-3$. But note that $\alpha_1\cap\pi_2=(K\cap\pi_1)\cap\pi_2=(K\cap\pi_1)\cap(K\cap\pi_2)=\alpha_1\cap\alpha_2$. Therefore we have dim$(\alpha_1\cap\alpha_2)=n-3$.
	\end{proof}

	\textbf{Step 8.} No hyperplane contains all the elements of $S$.

	\begin{proof}\renewcommand{\qedsymbol}{}
		Suppose to the contrary there is a hyperplane $\pi$ containing all the elements of $S$. Since $|S|>1$, $\pi$ is not an $a$-hyperplane. From Step $6$, we have $u_b(\alpha)=0$, for any $\alpha\in S$. Hence $\pi$ is not a $b$-hyperplane. So $\pi$ must intersect $K$ in $c$ points. Also we have $u_c(\alpha)=q$, for any $\alpha\in S$. Therefore, any $\alpha\in S$  is contained in $q-1$ $c$-hyperplanes different from $\pi$. Now counting all such hyperplanes, we have $t_c\ge (q-1)(\sqrt{q}^3+1)+1=\sqrt{q}^5-\sqrt{q}^3+q>q^2-\sqrt{q}^3+q$, a contradiction to Step 5.
	\end{proof}

	\textbf{Step 9.} $K$ is a cone.

	\begin{proof}\renewcommand{\qedsymbol}{}
		From Steps 7 and 8 it follows that all the elements of $S$ intersect in a unique $(n-3)$-dimensional space $\eta$ (a generalization of the fact that a set of mutually intersecting lines is either a set of lines through a point or a set of coplanar lines). Now counting all the points of $K$ contained in the elements of $S$, we see that $|S|(\theta_{n-2}-\theta_{n-3})+\theta_{n-3}=\sqrt{q}^{2n-1}+\theta_{n-2}=k$. Therefore $K$ is the pointset of $S$ which is a cone with vertex $\eta$. 
	\end{proof}

	\textbf{Step 10.} $K$ is an unital cone.

	\begin{proof}
		The number of hyperplanes through the vertex $\eta$ is $\theta_2$. Since $t_b=\theta_n-\theta_2>\theta_2$, there exists a $b$-hyperplane $\pi'$ not containing $\eta$. Let $\tau$ be a plane contained in $\pi'$ and disjoint from $\eta$ (this exists since $\eta\cap\pi'$ is $(n-4)$-dimensional). We prove the theorem by showing that $C:=K\cap \tau$ is a unital.
	
		Note that, for any $\alpha\in S$, $\alpha\cap \tau$ is a point. Hence $|C|=\sqrt{q}^3+1$. Note that $C$ doesn't contain any lines as otherwise there would exist hyperplanes intersecting $K$ in $\theta_{n-1}$ points, a contradiction. As $K$ is a line blocking set, $C$ is a line blocking set of $\tau$. Let $P$ be a point in $C$ and $\alpha_P$ be the unique element of $S$ such that $P=\tau\cap \alpha_P$. Let $\pi_P$ be the unique $a$-hyperplane containing $\alpha_P$. Then note that the line $\pi_P\cap \tau$ intersects $K$ in $K\cap(\pi_P\cap \tau)=\alpha_P\cap \tau=\{P\}$. Thus every point of $C$ is essential. In other words $C$ is a non-trivial minimal line blocking set of size $\sqrt{q}^3+1$ in $\tau$. Therefore, $C$ is an unital (see \cite{bruenthas}). 
	\end{proof}

	By Theorems \ref{innzu2} and \ref{unital}, we now have the following corollary.
	
	\begin{corollary}
		In $\Sigma=\PG(n,q)$, $n\ge3$ and $q$ a square prime power, an $(n-1)$-blocking set $K$ of type $(a,b,c)_{n-1}$, where $a=\theta_{n-2}, b=\theta_{n-3}+\sqrt{q}^{2n-3}$ and $c=\theta_{n-2}+\sqrt{q}^{2n-3}$, is a unital cone.
	\end{corollary}

\subsection{Hyperoval and maximal arc cones}
	Recall that for a cone $K$ in $\Sigma=\PG(n,q)$ with vertex $V$ all the possible bases are isomorphic to each other. If $V$ is $m$-dimensional with $m\ge1$ then any hyperplane $\pi$ of $\Sigma$ not containing $V$ intersects $K$ in a cone $K'$ with base isomorphic to that of $K$ and an $(m-1)$-dimensional vertex $V'=\pi\cap V$. Therefore all the hyperplanes not containing $V$ have the same number of points from $K$. The following lemma and corollary prove the converse which will be the key to proving the characterization theorems for hyperoval cones and maximal arc cones.  
	\begin{lemma}\label{lemcon}
		Let $P$ be a point in $\Sigma=\PG(n,q)$ and let $K$ be a set of points in $\Sigma$ such that every hyperplane of $\Sigma$ not containing $P$ intersects $K$ in a fixed number of points. Then $K\cup\{ P\}$ is a cone with $P$ as the vertex.
	\end{lemma}

	\begin{proof}
		Let $k$ be the number of points in $K$ and suppose that every hyperplane that does not contain $P$ has $t$ points of $K$. Let $\mu$ be an $(n-2)$-dimensional subspace not containing $P$. Let $|K\cap\mu|=x$ and 	$|K\cap\langle P,\mu\rangle|=y$. Counting the points of $K$ in hyperplanes through $\mu$, we see that $k=q(t-x)+y$. Now if $\mu'$ is another hyperplane of $\langle P,\mu\rangle$ disjoint from $P$ with $|K\cap\mu'|=x'$, then $k=q(t-x')+y$. Thus $x=x'$ and we see that all hyperplanes of $\langle P,\mu\rangle$, not containing $P$, contain the same number of points of $K$. Repeating this argument multiple times we see that for a point $Q\ne P$, every hyperplane of the line $PQ$ disjoint from $P$ (i.e. a point of $PQ$ different from $P$) contains the same number of points from $K$. Hence if $Q\in K$, then all the points on the line $ PQ$ are contained in $K\cup P$ and if $Q\notin K$ then $\langle P,Q\rangle\cap K=\{P\}$. Hence $K\cup \{P\}$ is a cone with $P$ as the vertex.
	\end{proof}

	\begin{corollary}\label{corcon}
		Let $V$ be an $m$-dimensional subspace of $\Sigma=\PG(n,q)$ and	let $K$ be the set of points in $\Sigma$ such that every hyperplane of $\Sigma$ that does not contain $V$ in a fixed number of points, then $K\cup V$ is a cone with $V$ as the vertex.
	\end{corollary}

	\begin{proof}
		Note that for any point $P\in V$, any hyperplane of $\Sigma$ not containing $P$ does not contain $V$. Hence from \Cref{lemcon} we have that $K\cup \{P\}$ is a cone with $P$ as the vertex. Thus we can see that for any point $Q$ in $K\setminus V$ the line $\langle P,Q\rangle$ is contained in $K\cup V$. As this is true for any point $P\in V$ and $Q\in \{K\setminus V\}$, we see that $K\cup V$ is a cone with $V$ as the vertex.
	\end{proof}

	The following theorem proves the characterization theorem for hyperoval cones in 3-dimensional projective spaces. Later we prove it for dimensions higher than $3$.

	\begin{theorem}\label{hyperO}
		In $\Sigma=\PG(3,q)$, a point set $K$ of type $(a,b,c)_2$, where $a=1,b=q+2$ and $c=2q+1$, is a cone over a hyperoval.
	\end{theorem}

	The proof of the theorem has been split into the following five steps.

	\textbf{Step 1.} $(q+1)|k$ and $q$ divides either $k-1$ or $k-2$.
	
	\begin{proof}\renewcommand{\qedsymbol}{}
		Putting the values of $a,b,c$ in \Cref{ta} we get
		\begin{multline}\label{ttt}
			t_a=\dfrac{1}{(q+1)(2q)}(k^2(q+1)-k(q+1+(3q+2)(q^2+q+1))+(q+2)(2q+1)(q^3+q^2+q+1).
		\end{multline}
		As $t_a$ is an integer we have $(q+1)|k$, and $q|(k-1)(k-2)$. As $k-1$ and $k-2$ are coprime and $p$ is the only prime dividing $q$ we have $q|(k-1)$ or $q|(k-2)$.	
	\end{proof}
	
	\textbf{Step 2.} $k\le 2q^2+1$.

	\begin{proof}\renewcommand{\qedsymbol}{}
		Let $\pi$ be a $1$-plane with $K\cap\pi=\{P\}$ and $l$ be a line through $P$ contained in $\pi$. Since each of the $q$ planes through $l$ different from $\pi$ has at most $2q$ points of $K\setminus\{P\}$, we have $k\le q(2q)+1$.
	\begin{comment}
		
		Counting these planes and all the points of $K$ through these planes we get
		\begin{equation*}
			u_a+u_b+u_c=q+1,
		\end{equation*}
		and 
		\begin{equation}\label{uuu}
			1+u_b(q+1)+u_c(2q)=k
		\end{equation}
		respectively with $u_a\ge 1.$
		
		$k$ attains its maximum value of $2q^2+1$ when $u_b=0,u_c=q$. 

	\end{comment}
	\end{proof}

	\textbf{Step 3.} $k=q^2+2q+1$.
	
	\begin{proof}\renewcommand{\qedsymbol}{}
		From Step 1, $k=\lambda(q+1)$. From Step 2 we have $c=2q+1\le k \le 2q^2+1$, so $2\le\lambda\le 2q-2$ for $q\ge3$ and $2\le\lambda\le3$ for $q=2$. 
		
		From Step 1 we also have $(q+1)\lambda-1=\xi q$ or $(q+1)\lambda-2=\nu q.$ Therefore $\lambda\equiv1$(mod $q$) or $\lambda\equiv2$(mod $q$). Hence $\lambda=2,q+1$ or $q+2$ resulting in $k=2q+2,k=q^2+2q+1$ or $k=q^2+3q+2$ respectively.
		
		If $k=2q+2$, $t_c=-(q^2+2q+1)/2<0$ for all values of $q$, a contradiction to the fact that $t_c\ge1$.
		
		If $k=q^2+3q+2$, then we have $$t_a=q/2+1$$ and $$t_b=\dfrac{q^4-q^3-4q^2-4q-1}{q-1}.$$ For $t_a$ and $t_b$ to be integers we must have $2|q$ and $(q-1)|9$. The only possible integer values being $q=2$ or $q=4.$ 
		
		If $q=2,$ $t_b=-17$, a contradiction to the fact that $t_b\ge1$.
		
		If $q=4$, consider a line $l'$ in a $1$-plane with $l'\cap K=\emptyset$. Let $u_i'$ denote the number of $i$-planes through $l'$. Counting these planes and all the points of $K$ through these planes we get 
		\begin{equation}\label{prob1}
			u_a'+u_b'+u_c'=q+1=5,
		\end{equation}  
		and 
		\begin{equation}\label{prob2}
			u_a'+6u_b'+9u_c'=k=30
		\end{equation}
		with $u_a'\ge1$. Since $u_a'\equiv0$(mod $3$) from \Cref{prob2}, we find $u_a'=3$ and it easily follows that there are no integer solution to the Equations (\ref{prob1}) and (\ref{prob2}) with $u_a'\ge1, u_b'\ge0$ and $u_c'\ge0$. Therefore $q=4$ is not a possible choice for $q$.
		
		We conclude that $k\ne 2q+2,k\ne q^2+3q+2$, so $k=q^2+2q+1$.
	\end{proof}
	
	\textbf{Step 4.} $K$ is a cone.
	
	\begin{proof}\renewcommand{\qedsymbol}{}
		Putting the given values of $a,b,c$ and $k=q^2+2q+1$ in Equations (\ref{ta}), (\ref{tb}) and (\ref{tc}) we get $$t_a=\dfrac{q^2-q}{2}$$$$t_b=q^3$$ and $$t_c=\dfrac{q^2+3q+2}{2}.$$ 
		
		Let $l$ and $\pi$ be as in Step 2 and $u_i$ denote the number of $i$-planes through $l$. Counting the points of $K$ through the planes containing $l$ we have
		\begin{equation}\label{uuu}
		1+u_b(q+1)+u_c(2q)=k=q^2+2q+1.
		\end{equation}
%From \Cref{uuu} we also have 
%\begin{equation}\label{uno}
%u_b(q+1)+u_c(2q)=q^2+2q.
%\end{equation}
		Considering \Cref{uuu} modulo $q$ we have $u_b\equiv 0$(mod $q$). Hence $u_b=0$ or $u_b=q$. 
		
		If $u_b=q, u_a=u_c=1/2$ from \Cref{uuu}, a contradiction. 
		
		Therefore $u_b=0, u_a=q/2,u_c=(q+2)/2$, i.e. through every line $l$ in $\pi$ with $l\cap K=\{P\}$ there are $q/2$ 1-planes and $(q+2)/2$ $c$-planes. 
		
		Note that there are $q+1$ different choices for $l$ in $\pi$. Counting all the $1$-planes through $P$ we see that there are $(q+1)(q/2-1)+1=({q^2-q})/{2}$ of them. Similarly there are $(q+1)((q+2)/2)=({q^2+3q+2})/{2}$ $c$-planes through $P$. Since $t_a=({q^2-q})/{2}$, $t_c=({q^2+3q+2})/{2}$ and $t_a+t_c=\theta_2$ is the total number of planes through $P$, it follows that none of the planes through $P$ intersect $K$ in $b$ points and any plane not containing $P$ is a $b$-plane. Hence by \Cref{lemcon}, $K$ is a cone with vertex $P$.
	\end{proof}
		
	\textbf{Step 5.} $K$ is a hyperoval cone.
	
	\begin{proof}
		As $K$ is a cone with vertex $P$, all $c$-planes intersect $K$ in a pair of lines through $P$. Let $\mu$ be a plane not passing through $P$. A $1$-plane intersects $\mu$ in a line containing zero points of $K$ and a $c$-plane intersects $\mu$ in a line containing two points of $K$. Thus $K\cap\mu$ is a set of $q+2$ points intersecting every line in either 0 or 2 points. Therefore $\mu$ is a hyperoval and $K$ is a hyperoval cone with vertex $P$. 
	\end{proof}
	
	\begin{theorem}\label{hyperO2}
		In $\Sigma=\PG(n,q)$, $n\ge4,$ a plane blocking set $K$ of type $(a,b,c)_{n-1}$, where $a=\theta_{n-3},b=\theta_{n-2}+q^{n-3},c=\theta_{n-2}+q^{n-2}$, is a cone with an $(n-3)$-space as vertex and a hyperoval as base.
	\end{theorem}

	The proof of the theorem has been split into the following six steps.
	
	\textbf{Step 1.} If $\pi$ is an $a$-hyperplane then $K\cap\pi$ is an $(n-3)$-dimensional subspace $V$ of $\Sigma$. 
	
	\begin{proof}\renewcommand{\qedsymbol}{}
		As $K$ is a plane blocking set in $\Sigma$, so is its intersection with any subspace. Since $|K\cap\pi|=\theta_{n-3}$, by \cite{bose}, $K\cap\pi$ is an $(n-3)$-dimensional subspace of $\Sigma$.
	\end{proof}

	\textbf{Step 2.} $k\le 2q^{n-1}+\theta_{n-3}$.
	
	\begin{proof}\renewcommand{\qedsymbol}{}
		Let $\pi$ and $V$ be as in Step 1. Let $\eta$ be an $(n-2)$-dimensional subspace of $\Sigma$ contained in $\pi$ and containing $V$. Since each of the $q$ hyperplanes through $\eta$ different from $\pi$ has at most $c-\theta_{n-3}=2q^{n-2}$ points of $K\setminus V$, we have $k\le q(2q^{n-2})+\theta_{n-3}=2q^{n-1}+\theta_{n-3}$.
%		Counting these hyperplanes and all the points of $K$ through these hyperplanes we get 
%		\begin{equation*}
%			u_a+u_b+u_c=q+1
%		\end{equation*}
%		and
%		\begin{equation}\label{uu2}
%			k=a+(b-a)u_b+(c-a)u_c
%		\end{equation}
%		respectively with $u_a\ge1$. 		
%		Therefore $k$ attains its maximum value of $k=\theta_{n-3}+(2q^{n-2})q=2q^{n-1}+\theta_{n-3}$ when $u_a=1,u_b=0,u_c=q$. 
	\end{proof}

	\textbf{Step 3.} $k\equiv\theta_{n-4}$(mod $q^{n-3}$).

	\begin{proof}\renewcommand{\qedsymbol}{}
		Note that $a\equiv b\equiv c\equiv \theta_n\equiv\theta_{n-1}\equiv\theta_{n-4}$(mod $q^{n-3}$). As $q$ is a prime power, $\theta_{n-4}$ and $q^{n-3}$ are coprime. Therefore by \Cref{equivl} we have $k\equiv\theta_{n-4}$(mod $q^{n-3}$).
	\end{proof}

	\textbf{Step 4.} $k=\theta_{n-1}+q^{n-2}, t_a=(q^2-q)/2, t_b=\theta_n-\theta_2,t_c=(q^2+3q+2)/2$.

	\begin{proof}\renewcommand{\qedsymbol}{}
		Note that $c\le k\le 2q^{n-1}+\theta_{n-3}$ from Step 2. Hence if we show that $k$ cannot assume any value in the intervals $[c,\theta_{n-1}+q^{n-2}-q^{n-3}]$ for all $n\ge4,q\ge2$ and $[\theta_{n-1}+q^{n-2}+q^{n-3},2q^{n-1}+\theta_{n-3}]$ for all $n\ge4,q\ge2$, by step 3, we can conclude that $k=\theta_{n-1}+q^{n-2}$. However, for $q=2$, we have $\theta_{n-1}+q^{n-2}=2q^{n-1}+\theta_{n-3}$. Therefore the interval $[\theta_{n-1}+q^{n-2}+q^{n-3},2q^{n-1}+\theta_{n-3}]$ is empty for $q=2$. So proving $k$ cannot assume values in $[\theta_{n-1}+q^{n-2}+q^{n-3},2q^{n-1}+\theta_{n-3}]$ for $n\ge4,q\ge3$ is sufficient.
	
		Recall that $t_a$ and $t_c$ are quadratic expressions in $k$ with positive coefficients for $k^2$ as $a<b<c$. Therefore if $t_a\le\lambda$ (similarly $t_c\le\lambda$) for two distinct values $x,y$ (with $x<y$) of $k$, then $t_a\le\lambda$ (similarly $t_c\le\lambda$) in the interval $[x,y]$.
	
		For $k=c,$
		\begin{align*}
			t_c=\dfrac{-q^{n+1}+q^3-2q+2}{2(q-1)^2},
		\end{align*}
		which is negative for all $n\ge4,q\ge2$.
	
		For $k=\theta_{n-1}+q^{n-2}-q^{n-3},$
		\begin{align*}
			t_c=\dfrac{-q^{n+1}-q^n+q^{n-1}+q^5+q^4-4q^3+5q-2}{2q(q-1)^2},
		\end{align*}
		which is negative for all $n\ge4,q\ge2$.
		
		Therefore $t_c$ is negative for $k$ in the interval $[c,\theta_{n-1}+q^{n-2}-q^{n-3}]$ for all $n\ge4,q\ge2$. Hence $k\notin[c,\theta_{n-1}+q^{n-2}-q^{n-3}]$. 
	
		For $k=\theta_{n-1}+q^{n-2}+q^{n-3},$
		\begin{align*}
			t_a=\dfrac{-q^{n-1}+\theta_{n-3}+q^3}{2(q+1)},
		\end{align*}
		which is $\le 1/2$ for all $n\ge4,q\ge2$.
	
		For $k=2q^{n-1}+\theta_{n-3},$
		\begin{align*}
			t_a=\dfrac{-q^{n}+\theta_{n-1}+2q^3-3q^2+q+1}{2(q+1)},
		\end{align*}
		which is negative for all $n\ge4,q\ge3$. 
		
		Therefore $t_a\le1/2$ for $k$ in the interval $[\theta_{n-1}+q^{n-2}+q^{n-3},2q^{n-1}+\theta_{n-3}]$ for all $n\ge 4, q\ge3$. Hence $k\notin[\theta_{n-1}+q^{n-2}+q^{n-3},2q^{n-1}+\theta_{n-3}]$ for all $n\ge4,q\ge3$. 
		
		Therefore we have shown that $k=\theta_{n-1}+q^{n-2}.$
		
		The values of $t_a,t_b$ and $t_c$ follow from the Equations (\ref{ta}), (\ref{tb}) and (\ref{tc}) respectively.
	\end{proof}

	\textbf{Step 5.} Any hyperplane containing $V$ intersects $K$ in either $a$ or $c$ points and any hyperplane not containing $V$ intersects $K$ in $b$ points.	

	\begin{proof}\renewcommand{\qedsymbol}{}
		Let $\eta$ and $\pi$ be as in Step 2 and let $u_i$ denote the number of $i$-hyperplanes through $\eta$. Counting the points of $K$ through the hyperplanes containing $\eta$, we get 
		\begin{align}\label{eta3}
			&a+(b-a)u_b+(c-a)u_c=k=\theta_{n-1}+q^{n-2} \nonumber\\
			\Rightarrow&(1+q)u_b+2qu_c=2q+q^2.
		\end{align}
		
%		Putting $k=\theta_{n-1}+q^{n-2}$ in \Cref{uu2} we get	
%		\begin{equation}\label{eta3}
%			(1+q)u_b+2qu_c=2q+q^2.
%		\end{equation}
		Considering \Cref{eta3} modulo $q$, we have 
		\begin{equation*}
			u_b\equiv0(\text{mod }q).
		\end{equation*}
		Therefore $u_b=0$ or $q$. 
		
		Putting $u_b=q$ gives $u_a=u_c=1/2$, a contradiction. 
		
		Hence $u_a=q/2, u_b=0 $ and $u_c=(q+2)/2$, i.e. through every $(n-2)$-dimensional subspace $\eta$ in $\pi$ with $\eta\cap K= V$ there are $q/2$ $a$-hyperplanes and $(q+2)/2$ $c$-hyperplanes.

		The rest of the proof follows similar to the proof of Step 4, \Cref{hyperO}, with $l$ and $P$ replaced by $\eta$ and $V$.
%		Now counting the number of hyperplanes containing $V$ and corresponding to each intersection number, by changing $\eta\in\pi$, we see that $(q^2-q)/2$ of them correspond to $a$, $(q^2+3q+2)/2$ correspond to $c$ and none correspond to $b$. 
%		
%		Reviewing the values of $t_a$ and $t_c$ we see that any hyperplane intersecting $K$ in either $a$ or $c$ contains $V$ and any hyperplane not containing $V$ corresponds to the intersection number $b$.
	\end{proof}

	\textbf{Step 6.} $K$ is a hyperoval cone.

	\begin{proof}
		From \Cref{corcon} it follows that $K$ is a cone with vertex $V$ and a plane $\tau$ containing $q+2$ points as base. 
		
		The rest of the proof follows similar to Step 5 of \Cref{hyperO}, with $P$ and $\mu$ replaced by $V$ and $\tau$ respectively.
%		 
%		Any $a$-hyperplane through $V$ intersects $\tau$ in a line with $0$ points of $K$ and vice versa, every line with $0$ points of $K$ in $\tau$ defines an $a$-hyperplane. Similarly any $c$-hyperplane through $V$ intersects $\tau$ in a line with $2$ points of $K$ and vice versa, every line with $2$ points of $K$ in $\tau$ defines a $c$-hyperplane. Thus $K\cap \tau$ is a set of $q+2$ points such that any line of $\tau$ contains 0 or 2 points of $K$. Hence $K\cap \tau$ is a hyperoval and $K$ is a cone with an $(n-3)$-dimensional vertex $V$ and a hyperoval as base.
	\end{proof}
	
	\begin{remark}
		Note that we get $u_a=q/2$ in Theorems \ref{hyperO} and \ref{hyperO2}, proving the well-known fact that hyperovals exist only if $q$ is even.
	\end{remark}

	The following theorem generalizes hyperoval cones to maximal arc cones but with added restrictions. In order to avoid dealing with trivial cases we assume that $2\le d\le q-1$ where $d$ is the degree of the maximal arc in the base of the cone. Although the proof of the theorem is similar to that of \Cref{hyperO2}, we do it separately to avoid dealing with the special cases formed by merging the proofs.
	
	\begin{theorem}\label{maximalA}
		In $\Sigma=\PG(n,q)$, $n\ge 5,$ a plane blocking set of type $(a,b,c)_{n-1}$, where $a=\theta_{n-3},b=q^{n-3}(qd+d-q)+\theta_{n-4},c=q^{n-2}d+\theta_{n-3}$ with $2\le d\le q-1$ and gcd$(d-1,q)=1$%$(d-1)\mid(q-1)$
		, is a cone with an $(n-3)$-dimensional space as vertex and a maximal arc of degree $d$ as base.
	\end{theorem}

	The proof of the theorem has been split into the following six steps.

	\textbf{Step 1.} If $\pi$ is an $a$-hyperplane then $K\cap\pi$ is an $(n-3)$-dimensional subspace $V$ of $\Sigma$. 
	
	\begin{proof}\renewcommand{\qedsymbol}{}
		As $K$ is a plane blocking set in $\Sigma,$ so is its intersection with any subspace. Since $|K\cap\pi|=\theta_{n-3}$, by \cite{bose}, $K\cap \pi$ is an $(n-3)$-dimensional subspace of $\Sigma$.
	\end{proof}
	
	\textbf{Step 2.} $k\le dq^{n-1}+\theta_{n-3}$.
	
	\begin{proof}\renewcommand{\qedsymbol}{}
		Let $\pi$ and $V$ be as in Step 1. Let $\eta$ be an $(n-2)$-dimensional subspace of $\Sigma$ contained inside $\pi$ and containing $V$. Since each of the $q$ hyperplanes through $\eta$ different from $\pi$ has at most $c-\theta_{n-3}=dq^{n-2}$ points of $K\setminus V$, we have $k\le q(dq^{n-2})+\theta_{n-3}=dq^{n-1}+\theta_{n-3}$.
%		
%		Let $\pi$ and $\eta$ be as in Step 1. Let $\eta$ be an $(n-2)$-dimensional subspace of $\Sigma$ contained inside $\pi$ and containing $V$. Let $u_i$ denote the number hyperplanes through $\eta$ intersecting $K$ in $i$ points. Counting these hyperplanes and all the points of $K$ through these hyperplanes we get
%		\begin{equation*}
%			u_a+u_b+u_c=q+1
%		\end{equation*}
%		and
%		\begin{equation}\label{uu3}
%			k=a+(b-a)u_b+(c-a)u_c.
%		\end{equation}
%		respectively with $u_a\ge1$.
%		
%		As $a<b<c, k$ attains a maximum value of $\theta_{n-3}+(dq^{n-2})q=dq^{n-1}+\theta_{n-3}$ when $u_a=1,u_b=0,u_c=q$. 
	\end{proof}
	
	\textbf{Step 3.} $k\equiv\theta_{n-4}$(mod $q^{n-3}$).

	\begin{proof}\renewcommand{\qedsymbol}{}
		Note that $a\equiv b\equiv c\equiv \theta_n\equiv\theta_{n-1}\equiv\theta_{n-4}$(mod $q^{n-3}$). As $q$ is a prime power, $\theta_{n-4}$ and $q^{n-3}$ are coprime. Hence by \Cref{equivl} we have $k\equiv\theta_{n-4}$(mod $q^{n-3}$).
	\end{proof}

	\textbf{Step 4.} $k=q^{n-2}(qd+d-q)+\theta_{n-3}, t_a=q(q+1-d)/d, t_b=\theta_n-\theta_2$ and $t_c=(q+1)(qd+d-q)/d$.
	
	\begin{proof}\renewcommand{\qedsymbol}{}
		Recall that $c\le k\le dq^{n-1}+\theta_{n-3}$. Therefore if we show that $k$ cannot assume any value in the intervals $[c,q^{n-2}(qd+d-q)+\theta_{n-3}-q^{n-3}]$ and $[q^{n-2}(qd+d-q)+\theta_{n-3}+q^{n-3},dq^{n-1}+\theta_{n-3}]$ for all $n\ge5,q\ge2$, by Step 3 we can say that $k=q^{n-2}(qd+d-q)+\theta_{n-3}$.
		
		Recall that $a<b<c$, $t_a$ and $t_c$ are quadratic expressions in $k$ with a positive leading coefficient. Therefore if $t_a\le\lambda$ (similarly $t_c\le\lambda$) for two distinct values $x,y$ (with $x<y$) of $k$, then $t_a\le\lambda$ (similarly $t_c\le\lambda$) in the interval $[x,y]$.

	For $k=c,$
	\begin{equation*}
		t_c=\dfrac{(d-1)^2q^{n+1}-q^3+qd-d^2+d}{(q-1)d(d-q-1)}
	\end{equation*}
	which is negative for all $n\ge 2$.%,q\ge 2$ and $2\le d\le q-1$.

	For $k=q^{n-2}(qd+d-q)+\theta_{n-3}-q^{n-3},$
	\begin{multline*}
		t_c=\dfrac{1}{(q-1)dq(d-q-1)}(q^{n+1}(d-1)+q^n(d-1)-q^{n-1}+d^2(q^4+q^3-q^2-q)\\+d(-q^5-3q^4+q^2+1)+(q^5+q^4+2q^2-q))
	\end{multline*}
	which is negative for all $n\ge 4,q\ge 2$ and $2\le d\le q-1$.
	
	Therefore $t_c$ is negative for $k$ in the interval $[c,q^{n-2}(qd+d-q)+\theta_{n-3}-q^{n-3}]$ for all $n\ge 5,q\ge 2$ and $2\le d\le q-1$. Hence $k\notin[c,q^{n-2}(qd+d-q)+\theta_{n-3}-q^{n-3}]$.

	For $k=q^{n-2}(qd+d-q)+\theta_{n-3}+q^{n-3},$
	\begin{multline*}
		t_a=\dfrac{1}{dq(q+1)(d-1)}(-q^n-q^{n-1}+\theta_{n-1}(d-1)-d^2q^2(q+1)+d(q^4+3q^3+2q^2-2q-2)\\+(-q^4-2q^3+4q+3))
	\end{multline*}
	which is negative for all $n\ge 5,q\ge 2$ and $2\le d\le q-1$.

	For $k=q^{n-1}d+\theta_{n-3},$
	$$t_a=\dfrac{q^2\theta_{n-3}(d-q)-(q^2-q-1)d^2+d(q^3-q-1)}{d(q+1)(d-1)}$$
	which is less than 1 for all $n\ge 4,q\ge 2$ and $2\le d\le q-1$.
	
	Therefore $t_a<1$ for $k$ in the interval $[q^{n-2}(qd+d-q)+\theta_{n-3}+q^{n-3},dq^{n-1}+\theta_{n-3}]$ for all $n\ge 5,q\ge 2$ and $2\le d\le q-1$. Hence $k\notin[q^{n-2}(qd+d-q)+\theta_{n-3}+q^{n-3},dq^{n-1}+\theta_{n-3}]$.
	
	The values of $t_a,t_b$ and $t_c$ follow from the Equations \ref{ta}, \ref{tb} and \ref{tc} respectively.  
	\end{proof}

	\textbf{Step 5.}  Any hyperplane containing $V$ intersects $K$ in either $a$ or $c$ points and any hyperplane not containing $V$ intersects in $b$ points.
	
	\begin{proof}\renewcommand{\qedsymbol}{}
		Let $\eta$ and $\pi$ be as in Step 2 and let $u_i$ denote the number of $i$-hyperplanes through $\eta$. Counting the points of $K$ through the hyperplanes containing $\eta$, we get
		\begin{align}\label{uu3}
			&a+(b-a)u_b+(c-a)u_c=k\nonumber\\
			\Rightarrow&\theta_{n-3}+q^{n-3}(d-1)(q+1)u_b+q^{n-2}du_c=\theta_{n-3}+q^{n-2}(qd+d-q)\nonumber\\
			\Rightarrow&(d-1)(q+1)u_b+qdu_c=q(qd+d-q).
		\end{align}
%		
%
%		
%		Let $\eta$ and $u_i$ be as in Step 2. Putting $k=\theta_{n-1}+q^{n-2}$ in \Cref{uu3} we get	
		
%		\begin{align*}
%			&\theta_{n-3}+q^{n-3}(d-1)(q+1)u_b+q^{n-2}du_c=\theta_{n-3}+q^{n-2}(qd+d-q)\\
%			\Rightarrow&(d-1)(q+1)u_b+qdu_c=q(qd+d-q).
%		\end{align*}
		As %$d-1\mid q-1$ and $q$ is a prime power, 
		$d-1$ and $q+1$ are coprime to $q$, from \Cref{uu3}, we have $q\mid u_b$. Thus $u_b=0$ or $q$.
	
		If $u_b=q,$ $u_c=0$ as $u_a\ge1$. Thus $$q(d-1)(q+1)=q(qd+d-q),$$ a contradiction. 
		
		Therefore $u_b=0$, $u_c=q+1-q/d$ and $u_a=q/d$, i.e. through every $(n-2)$-dimensional subspace $\eta$ in $\pi$ with $\eta\cap K=V$ there are $q/d$ $a$-hyperplanes and $(q+1-q/d)$ $c$-hyperplanes.
		
%Now counting the hyperplanes through $V$, by changing $\eta\in\pi$, we see that there are $(q/d-1)(q+1)$ $a$-hyperplanes, 0 $b$-hyperplanes and $(q+1-q/d)(q+1)$ $c$-hyperplanes. 
		Note that there are $q+1$ choices of $\eta$ in $\pi$. Counting all the $a$-hyperplanes through $V$ by changing $\eta$ in $\pi$ we see that there are $(q/d-1)(q+1)+1=t_a$ of them. Similarly there are $(q+1-q/d)(q+1)=t_c$ $c$-hyperplanes through $V$. As $t_a+t_c=\theta_2$ is the total number of hyperplanes through $V$, we conclude that any hyperplane not containing $V$ is a $b$-hyperplane.
% $a$-hyperplanes and $c$-hyperplane passes through $V$ and  
%		Reviewing the values of $t_a$ and $t_c$, we see that the hyperplanes corresponding to $a$ and $c$ contain $V$ and any hyperplane not containing $V$ corresponds to $b$. 
	\end{proof}

	\textbf{Step 6.} $K$ is a maximal arc cone.

	\begin{proof}
		From Corollary \ref{corcon} and Step 5, it follows that $K$ is a cone with vertex $V$ and a plane $\tau$ containing $qd+d-q$ points as base. So any $a$-hyperplane through $V$ intersects $\tau$ in a line with $0$ points of $K$ and any $c$-hyperplane through $V$ intersects $\tau$ in a line with $d$ points of $K$ and vice versa. Therefore $K\cap \tau$ is a set of $qd+d-q$ points such that any line of $\tau$ contains $0$ or $d$ points of $K$. Hence $K\cap \tau$ is a maximal arc and $K$ is a cone with an $(n-3)$-dimensional vertex $V$ and a maximal arc of degree $d$ as base.
	\end{proof}

	\begin{remark}
		Note that $u_a=q/d$ in \Cref{maximalA}, proving the simple fact that for maximal arcs with $d\le q$ to exist, it is necessary that $d|q$.
	\end{remark}

\textbf{Acknowledgments:} The author would like to thank his supervisors Dr. Geertrui Van de Voorde and Prof. Bart De Bruyn for the insightful discussions, invaluable remarks and suggestions.

\end{document}